\documentclass[11pt,a4paper,notitlepage]{article}
\usepackage{amsmath}
\usepackage[affil-it]{authblk}
\usepackage{amsthm}
\usepackage{amssymb}
\usepackage{enumerate}
\usepackage[page]{appendix}
\usepackage{geometry}
\usepackage{tabularx}
\usepackage{multirow}
\newtheorem{Lemma}      {Lemma} [section]
\newtheorem{Theorem}    [Lemma] {Theorem}

\numberwithin{equation}{section}

\newcommand{\diag}{\mathop{\mathrm{diag}}}

\newcommand{\Aut}{\mathop{\mathrm{Aut}}}
\newcommand{\Out}{\mathop{\mathrm{Out}}}
\newcommand{\Soc}{\mathop{\mathrm{Soc}}}

\numberwithin{equation}{section}
\begin{document}
\title{Generating minimally transitive permutation groups}
\author{Gareth M. Tracey%
\thanks{Electronic address: \texttt{G.M.Tracey@warwick.ac.uk}}} 
\affil{Mathematics Institute, University of Warwick,\\Coventry CV4 7AL, United Kingdom}
\date{June 13, 2015}
\maketitle
\begin{abstract} We improve the upper bounds (in terms of $n$) in \cite{LucMin} and \cite{ShepWie} on the minimal number of elements required to generate a minimally transitive permutation group of degree $n$.\end{abstract}
\section{Introduction}
A transitive permutation group $G\le S_{n}$ is called \emph{minimally transitive} if every proper subgroup of $G$ is intransitive. In this paper, we consider the minimal number of elements $d(G)$ required to generate such a group $G$, in terms of its degree $n$. For a prime factorisation $n=\prod_{p\text{ prime}} p^{n(p)}$ of $n$, we will write $\omega{(n)}:=\sum_{p} n(p)$ and $\mu(n):=\max\left\{n(p)\text{ : }p\text{ prime}\right\}$.

The question of bounding $d(G)$ in terms of $n$ was first considered by Shepperd and Wiegold in \cite{ShepWie}; there, they prove that every minimally transitive group of degree $n$ can be generated by $\omega(n)$ elements. It was then suggested by Pyber (see \cite{Pyber}) to investigate whether or not $\mu(n)+1$ elements would always suffice. A. Lucchini gave a partial answer to this question in \cite{LucMin}, proving that: \emph{if $G$ is a minimally transitive group of degree $n$, and $\mu(n)+1$ elements are not sufficient to generate $G$, then $\omega(n)\ge 2$ and $d(G)\le \lfloor \log_{2}(\omega{(n)}-1)+3\rfloor$.} 

In this note, we offer a complete solution to the problem, proving
\begin{Theorem}\label{Min} Let $G$ be a minimally transitive permutation group of degree $n$. Then $d(G)\le \mu(n)+1$.\end{Theorem}
Our approach follows along the same lines as Lucchini's proof of the main theorem in \cite{LucMin}. Indeed, his methods suffice to prove Theorem \ref{Min} in the case when a minimal normal subgroup of $G$ is abelian. Thus, our main efforts will be concerned with the case when a minimal normal subgroup of $G$ is a direct product of isomorphic nonabelian simple groups. The key step in this direction is Lemma \ref{SimplePrimes}, which we prove in Section 3. We use Section 2 to outline the method of \emph{crown-based powers} due to Lucchini and F. Dalla Volta; this will serve as the basis for our arguments. Finally, we prove Theorem \ref{Min} in Section 4.    
\section{Crown-based powers}
In this section, we outline an approach to study the minimal generation of finite groups, which is due to F. Dalla Volta and A. Lucchini. So let $G$ be a finite group, with $d(G)=d>2$, and let $M$ be a normal subgroup of $G$, maximal with the property that $d(G/M)=d$. Then $G/M$ needs more generators than any proper quotient of $G/M$, and hence, as we shall see below, $G/M$ takes on a very particular structure.  

We describe this structure as follows: let $L$ be a finite group, with a unique minimal normal subgroup $N$. If $N$ is abelian, then assume further that $N$ is complemented in $L$. Now, for a positive integer $k$, set $L_{k}$ to be the subgroup of the direct product $L^{k}$ defined as follows
$$L_{k}:=\left\{(x_{1},x_{2},\hdots,x_{k})\text{ : }x_{i}\in L\text{, }Nx_{i}=Nx_{j}\text{ for all }i,j\right\}$$
Equivalently, $L_{k}:=\diag(L^{k})N^{k}$, where $\diag{(L^{k})}$ denotes the diagonal subgroup of $L^{k}$. The group $L_{k}$ is called the \emph{crown-based power of $L$ of size $k$}. 

We can now state the theorem of Dalla Volta and Lucchini.
\begin{Theorem}[{\bf\cite{DVLuc}, Theorem 1.4}]\label{CrownTheorem} Let $G$ be a finite group, with $d(G)\ge 3$, which requires more generators than any of its proper quotients. Then there exists a finite group $L$, with a unique minimal normal subgroup $N$, which is either nonabelian or complemented in $L$, and a positive integer $k\ge 2$, such that $G\cong L_{k}$.\end{Theorem}

It is clear that, for fixed $L$, $d(L_{k})$ increases with $k$. To use this result, however, we will need a bound on $d(L_{k})$, in terms of $k$. This is provided by the next two theorems. Before giving the statements, we require some additional notation: for a group $G$ and a normal subgroup $M$ of $G$, let $P_{G,M}(d)$ denote the conditional probability that $d$ randomly chosen elements of $G$ generate $G$, given that their images modulo $M$ generate $G/M$. 
\begin{Theorem}[{\bf\cite{LucMin}, Theorem 2.1 and \cite{DVLuc}, Theorem 2.7}]\label{CrownBound} Let $L$ be a finite group with a unique minimal normal subgroup $N$ which is either nonabelian or complemented in $L$, and let $k$ be a positive integer. Assume also that $d(L)\le d$. Then \begin{enumerate}[(i)] 
\item If $N$ is abelian, then $d(L_{k})\le \max\left\{d(L),k+1\right\}$;
\item If $N$ is nonabelian, then $d(L_{k})\le d$ if and only if $k\le P_{L,N}(d)|N|^{d}/|C_{\Aut{(N)}}(L/N)|$.
\end{enumerate}\end{Theorem}
We will also need an estimate for $P_{L,N}(d)$.
\begin{Theorem}[{\bf\cite{DetLuc}, Theorem 1.1}]\label{53Over90} Let $L$ be a finite group, with a unique minimal normal subgroup $N$, which is nonabelian, and suppose that $d\ge d(L)$. Then $P_{L,N}(d)\ge 53/90$.\end{Theorem}
  
\section{Indices of proper subgroups in finite simple groups}
Before stating and proving the main result of this section, we need some standard notation: for a positive integer $m$, $\pi(m)$ denotes the set of prime divisors of $m$. Our lemma can now be stated as follows. 
\begin{Lemma}\label{SimplePrimes} Let $S$ be a nonabelian simple group. Then there exists a set of primes $\Gamma=\Gamma(S)$ with the following properties:\begin{enumerate}[(i)]
\item $|\Gamma|\le f(S)$, where $f(S):=r/2+1$ if $S$ is an alternating group of degree $r$, and $f(S):=4$ otherwise;
\item $\pi(|S:H|)$ intersects $\Gamma$ nontrivially for every proper subgroup $H$ of $S$.\end{enumerate}\end{Lemma}
\begin{proof} If $S=L_{2}(p)$, for some prime $p$, then since every maximal subgroup $M$ of $S$ has index divisible by either $p$ or $p+1$ (see \cite{Dickson}, for example), the result is clear. If $S=L_{2}(8)$, $L_{3}(3)$, $U_{3}(3)$ or $Sp_{4}(8)$, then direct computation using MAGMA (or Tables 8.1 to 8.6 and Table 8.14 in \cite{DerekColvaJohn}), implies that each maximal subgroup of $S$ has index divisible by at least one of the primes in $\left\{2,3\right\}$, $\left\{2,13\right\}$, $\left\{3,7\right\}$, and $\left\{2,3\right\}$, respectively.

Next, assume that $S=A_{r}$ is an alternating group of degree $r$, and let $p$ and $q$ be the two largest primes not exceeding $r$, where $p>q$. If $r=p$, then we can take $\Gamma:=\left\{r,q\right\}$, by Theorem 4 of \cite{LiePraeSax}. So assume that $p<r$, and for each $k$ in $p\le k\le r-1$, choose a prime divisor $q_{k}$ of $\binom{r}{k}$. Then set $\Gamma:=\Gamma(A_{r})=\left\{q_{p},\hdots,q_{r-1}\right\}\cup\left\{p,q\right\}$. We claim that $\Gamma$ satisfies (i) and (ii). To see this, note that $|\Gamma|\le r-p+2$, which is less than $r/2+2$ by Bertrand's postulate. This proves (i). To see that (ii) holds, let $H$ be a proper subgroup of $A_{r}$. If $p$ or $q$ does not divide $|H|$ then we are done, so assume that $pq$ divides $|H|$. Then $A_{k}\unlhd H\le S_{k}\times S_{r-k}$, for some $k$ with $p\le k\le r-1$, by Theorem 4 of \cite{LiePraeSax}. Hence, $H$ has index divisible by $\binom{r}{k}$, and (ii) follows.

So assume that $S$ is not one of the simple groups considered in the first two paragraphs above, and let $\Pi=\Pi(S)$ be the set of prime divisors of $|S|$ discussed in Corollary 6 of \cite{LiePraeSax}, so that $|\Pi|\le 3$. If $S$ does not occur in the left hand column of Table 10.7 in \cite{LiePraeSax}, then $\Gamma:=\Pi$ satisfies the conclusion of the lemma, by Corollary 6 of \cite{LiePraeSax}, so assume otherwise. 

Then $S$ is one of the simple groups in the left hand column of Table 10.7 in \cite{LiePraeSax}; we need to prove that there exists a set $\Gamma$ as in the statement of the lemma. If $H<S$ is not one of the exceptions listed in the middle column of Table 10.7, then $|S:H|$ intersects $\Pi$ non-trivially. Thus, all we need to prove is that there exists a prime $p$ such that, whenever $H$ is one of these exceptional subgroups, then $p$ divides $|S:H|$. Indeed, in this case, $\Gamma:=\Pi\cup \left\{p\right\}$ gives us what we need. 

So let $H$ be one of these subgroups. We consider each of the possibilities from Table 10.7 of \cite{LiePraeSax}:\begin{enumerate}
\item $S=PSp_{2m}(q)$ ($m$, $q$ even) or $P\Omega_{2m+1}(q)$ ($m$ even, $q$ odd), and $\Omega^{-}_{2m}(q) \unlhd H$. Then $H\le N_{S}(\Omega^{-}_{2m}(q))$, so $|S:N_{S}(\Omega^{-}_{2m}(q))|$ divides $|S:H|$. But $|N_{S}(\Omega^{-}_{2m}(q)):\Omega^{-}_{2m}(q)|\le 2$ using Corollary 2.10.4 part (i) and Table 2.1.D of \cite{LieKleid} and, for each of the two choices of $S$, we have $|S:\Omega^{-}_{2m}(q)|=q^{m}(q^{m}-1)$. Choosing $p$ so that $q=p^{f}$ now works.
 \item $S=P\Omega^{+}_{2m}(q)$ ($m$ even, $q$ odd), and $\Omega_{2m-1}(q) \unlhd H$. As above, $H\le N_{S}(\Omega_{2m-1}(q))$, and we use Corollary 2.10.4 part (i) and Table 2.1.D of \cite{LieKleid} to conclude that $|N_{S}(\Omega_{2m-1}(q)):\Omega_{2m-1}(q)|\le 2$. It follows that $\frac{1}{2}q^{m-1}(q^{m}-1)=|S:\Omega_{2m-1}(q)|$ divides $2|S:H|$. Since $m\ge 4$, choosing $p$ so that $q=p^{f}$ again works.
\item $S=PSp_{4}(q)$ and $PSp_{2}(q^{2}) \unlhd H$. Then $H\le N_{S}(\Omega_{2m-1}(q))$, and Corollary 2.10.4 part (i) and Table 2.1.D of \cite{LieKleid} gives $|N_{S}(PSp_{2}(q^{2})):PSp_{2}(q^{2})|\le 2$. It follows that $q^{2}(q^{2}-1)=|S:PSp_{2}(q^{2})|$ divides $2|S:H|$. Again, the prime $p$ satisfying $q=p^{f}$, for some $f$, gives us what we need.
\item In each of the remaining cases (see Table 10.7 in \cite{LieKleid}), we are given a tuple ($S$, $Y_{1}$,$\hdots$, $Y_{t(S)}$), where $t(S)\le 4$, $S$ is one of $L_{2}(8)$, $L_{3}(3)$, $L_{6}(2)$, $U_{3}(3)$, $U_{3}(3)$, $U_{3}(5)$, $U_{4}(2)$, $U_{4}(3)$, $U_{5}(2)$, $U_{6}(2)$, $PSp_{4}(7)$, $PSp_{4}(8)$, $PSp_{6}(2)$, $P\Omega^{+}_{8}(2)$, $G_{2}(3)$, $^{2}F_{4}(2)'$, $M_{11}$, $M_{12}$, $M_{24}$, $HS$, $M_{c}L$, $Co_{2}$ or $Co_{3}$, $Y_{i}<S$ for each $1\le i\le t(S)$, and $H$ is contained in at least one of the groups $Y_{i}$. In each case, we can easily see that there is a prime $p$, with $p$ dividing $|S:Y_{i}|$ for each $i$ in $1\le i\le t(S)$.\end{enumerate}
This completes the proof.\end{proof}

\section{The proof of Theorem \ref{Min}}
Before proceeding to the proof of Theorem \ref{Min}, we need three lemmas.
\begin{Lemma}\label{MinTransLemma}  Let $G$ be a transitive subgroup of $S_{n}$ ($n\ge 1$), let $1\neq M$ be a normal subgroup of $G$, and let $\Omega$ be the set of $M$-orbits. Then\begin{enumerate}[(i)]
\item Either $M$ is transitive, or $\Omega$ forms a system of blocks for $G$. In particular, the size of an $M$-orbit divides $n$. \item $|\Omega|=|G:AM|$, where $A$ is a point stabiliser in $G$.
\item If $G$ is minimally transitive, then $G^{\Omega}$ acts minimally transitively on $\Omega$.\end{enumerate}\end{Lemma}
\begin{proof} Part (i) is clear, so we prove (ii): if $M$ is transitive, then $AM=G$, so $|\Omega|=1=|G:AM|$. Otherwise, part (i) implies that the size of each $M$-orbit is $|M:M\cap A|=|AM:A|$, so the number of $M$-orbits is $n/|AM:A|=|G:AM|$. Part (ii) follows. Finally, part (iii) is Theorem 2.4 in \cite{DallaVolta}.\end{proof}

\begin{Lemma}[{\bf\cite{LucMor}, Proof of Lemma 3}]\label{CAutN} Let $L$ be a finite group with a unique minimal normal subgroup $N$, which is nonabelian, and write $N\cong S^{t}$, where $S$ is a nonabelian simple group. Then $|C_{\Aut{(N)}}(L/N)|\le t|S|^{t}|\Out{(S)}|$.\end{Lemma}

\begin{Lemma}[{\bf\cite{LieShal}, Proposition 4.4}]\label{SOutS} Let $S$ be a nonabelian finite simple group. Then $|\Out{(S)}|\le |S|^{1/4}$.\end{Lemma}

The preparations are now complete.
\begin{proof}[Proof of Theorem \ref{Min}] Assume that the theorem is false, and let $G$ be a counterexample of minimal degree. Also, let $A$ be the stabiliser in $G$ of a point $\alpha$, and let $m:=\mu(n)+1$. 

First, we claim that $G$ needs more generators than any proper quotient of $G$. To this end, let $M$ be a normal subgroup of $G$, and let $K$ be the kernel of the action of $G$ on the set of $M$-orbits. Then $G/K$ is minimally transitive of degree $s:=|G:AM|$, by Lemma \ref{MinTransLemma}, and hence, since $s$ divides $n$, the minimality of $G$ implies that there exists elements $x_{1}$, $x_{2}$, $\hdots$, $x_{m}$ in $G$ such that $G=\langle x_{1},x_{2},\hdots,x_{m},K\rangle$. But then $H:=\langle x_{1},x_{2},\hdots,x_{m}\rangle$ acts transitively on the set of $M$-orbits, so $HM=G$ by minimal transitivity of $G$. Hence $d(G/M)\le m$, which proves the claim.

Hence, by Theorem \ref{CrownTheorem}, $G\cong L_{k}$, for some $k\ge 2$, and some group $L$ with a unique minimal normal subgroup $N$, which is either nonabelian, or complemented in $L$. We now fix some notation: write $\Soc{(G)}=N_{1}\times N_{2}\times \hdots\times N_{k}$, where each $N_{i}\cong N\cong S^{t}$, for some simple group $S$, and $t\ge 1$, and set $X_{i}:=N_{1}\times N_{2}\times\hdots\times N_{i}$. We will also write $X_{0}:=1$, $H_{i+1}=N_{i+1}\cap X_{i}A$, and we denote by $\Delta_{i}$ the $X_{i}$-orbit containing $\alpha$, for $0\le i\le k$. Then $|\Delta_{i}|=n|X_{i}A|/|G|$ by Lemma \ref{MinTransLemma} part (ii), and hence
$$\frac{|\Delta_{i+1}|}{|\Delta_{i}|}=\frac{|X_{i+1}A|}{|X_{i}A|}=\frac{|N_{i+1}X_{i}A|}{|X_{i}A|}=|N_{i+1}:H_{i+1}|$$
Furthermore, it is shown in the proof of the main theorem in \cite{LucMin}, that $|\Delta_{i+1}|/|\Delta_{i}|=|N_{i+1}:H_{i+1}|$ is greater than $1$ for $0\le i\le k-2$, and also for $i=k-1$ if $N$ is abelian. Note also that $G/\Soc{(G)}\cong L/M$ is $m$-generated, by the previous paragraph; thus, $L$ is $m$-generated (see \cite{Luc3}).

We now separate the cases of $N$ being abelian or nonabelian. If $N$ is abelian, then $N\cong C_{p}^{t}$, for some prime $p$, so by the previous paragraph, $p$ divides $|N_{i+1}:H_{i+1}|=|\Delta_{i+1}|/|\Delta_{i}|$ for each $0\le i\le k-1$. Thus, $p^{k}$ divides $|\Delta_{k}|$, and hence divides $n$, by Lemma \ref{MinTransLemma} part (i). It follows that $k\le \mu(n)$, which, by Theorem \ref{CrownBound} part (i), contradicts our assumption that $d(G)>\mu(n)+1$.

Thus, $N$ is nonabelian. Hence, by the third paragraph, for each $i$ in $0\le i\le k-2$, $N_{i+1}$ has a direct factor $S_{i+1}$ ($S_{i+1}\cong S$), with $|S_{i+1}:S_{i+1}\cap H_{i+1}|>1$. Let $\Gamma=\Gamma(S)$ be the set of primes in Lemma \ref{SimplePrimes}, so that $|\Gamma|\le f(S)$, where $f(S)$ is as defined in Lemma \ref{SimplePrimes}. Then Lemma \ref{SimplePrimes} implies that for each $0\le i\le k-2$, the index $|S_{i+1}:S_{i+1}\cap H_{i+1}|$, and hence $|\Delta_{i+1}|/|\Delta_{i}|=|N_{i+1}:H_{i+1}|$, is divisible by some prime $p_{i+1}$ in $\Gamma$. 

So we now have a list of primes $p_{1}$, $p_{2}$, $\hdots$, $p_{k-1}$, with each $p_{i}$ in $\Gamma$, such that the product  $\prod_{i=1}^{k-1} p_{i}$ divides $|\Delta_{k-1}|$. For each prime $p$ in $\Gamma$, let $a_{(p)}$ be the number of times that $p$ occurs in this product. Then, since $|\Delta_{k-1}|$ divides $n$ by Lemma \ref{MinTransLemma} (i), $\prod_{p\in \Gamma} p^{a_{(p)}}$ divides $n$. Since $|\Gamma|\le f(S)$, and $\sum_{p\in \Gamma} a_{(p)}=k-1$, we have $a_{(p)}\ge (k-1)/f(S)$ for at least one prime $p$ in $\Gamma$. Hence, $(k-1)/f(S)\le \mu(n)$, and it follows that
\begin{align}k\le f(S)\mu(n)+1 &\le \frac{53|S|^{t\mu(n)}}{90t|\Out{(S)}|}\\
&\le \frac{53|N|^{m}}{90|C_{\Aut{(N)}}(L/N)|} &\text{( by Lemma \ref{CAutN})}\\
&\le \frac{P_{L,N}(m)|N|^{m}}{|C_{\Aut{(N)}}(L/N)|} &\text{( by Theorem \ref{53Over90})}\end{align}
The inequality at (4.1) above follows easily when $S$ is an alternating group of degree $r$, since $|S|=r!/2$, and $|\Out({S})|\le 4$ in this case (also, $|\Out({S})|\le 2$ if $r\neq 6$). It also follows easily when $S$ is not an alternating group, using Lemma \ref{SOutS}. Now, by Theorem \ref{CrownBound} part (ii), the inequality at (4.3) contradicts our assumption that $d(G)>m$. This completes the proof.\end{proof}  

\emph{Acknowledgments:} The author is hugely grateful to his supervisor Professor D.F. Holt for his careful reading of the paper, and to the Engineering and Physical Sciences Research Council for their continued support.        

\end{document}